\documentclass[11pt, a4paper]{article}

\usepackage{amsmath}
\usepackage{amsthm}
\usepackage{amssymb}
\usepackage{latexsym}
\usepackage{graphicx}
\usepackage{setspace}
\usepackage{float}
\usepackage{enumerate}
\usepackage[normalem]{ulem}
\usepackage{xcolor}
\usepackage{cite}
\usepackage{enumitem}

\newtheorem{theorem}{Theorem}
\newtheorem{lemma}{Lemma}
\newtheorem{corollary}{Corollary}
\newtheorem{definition}{Definition}
\newtheorem{proposition}{Proposition}

\theoremstyle{definition}
\newtheorem{example}{Example}
\theoremstyle{remark}

\newcommand{\OO}{\mathcal O}
\newcommand{\GG}{\mathcal G}
\newcommand{\EE}{\mathcal E}

\newcommand{\VV}{\mathcal V}
\newcommand{\MM}{\mathcal M}
\newcommand{\WW}{\mathcal W}
\newcommand{\LV}{\mathcal{L}(\mathcal V)}

\newcommand{\GL}{\mathcal{G\>\!\!L}}

\newcommand{\bR}{\overline{\mathbb{R}}}
\newcommand{\bx}{\bar x}
\newcommand{\AutV}{\operatorname{Aut}(\VV)}
\newcommand{\DerV}{\operatorname{Der}(\VV)}

\newcommand{\wpartial}{\widehat{\partial}}

\newcommand{\abs}[1]{\left\vert #1 \right\vert}
\newcommand{\norm}[1]{\left\Vert #1 \right\Vert}
\newcommand{\Norm}[1]{\left\vert \mkern -2mu \left\Vert #1 \right\Vert \mkern -2mu \right\vert}
\newcommand{\ip}[2]{\left< #1, #2 \right>}
\newcommand{\set}[2]{\left\lbrace #1 \, : \, #2 \right\rbrace}
\newcommand{\Rn}{\mathbb{R}^n}

\begin{document}

\title{Commutation principles for nonsmooth variational problems on Euclidean Jordan algebras}

\author{Juyoung Jeong\footnote{Department of Mathematics, Soongsil University, Seoul 06978, Republic of Korea. Partially supported by the National Research Foundation of Korea NRF-2021R1C1C2008350.} \and David Sossa\footnote{Universidad de O'Higgins, Instituto de Ciencias de la Ingenier\'ia, Av.\,Libertador Bernardo O'Higgins 611, Rancagua, Chile (e-mail: david.sossa@uoh.cl). Partially supported by  FONDECYT (Chile)  through grant 11220268, and MATH-AMSUD 23-MATH-09 MORA-DataS project.}    } 

\date{\today}

\maketitle

\bigskip
\bigskip

\begin{quote}{\small 
    \textbf{Abstract}. 
    
    The commutation principle proved by Ram\'irez, Seeger, and Sossa (SIAM J Optim 23:687–694, 2013) in the setting of Euclidean Jordan algebras says that for a Fr\'echet differentiable function $\Theta$ and a spectral function $F$, any local minimizer or maximizer $a$ of $\Theta+F$ over a spectral set $\EE$ operator commutes with the gradient of $\Theta$ at $a$. In this paper, we improve this commutation principle by allowing $\Theta$ to be nonsmooth. For example, for the case of local minimizer, we show that $a$ operator commutes with some element of the limiting (Mordukhovich) subdifferential of $\Theta$ at $a$ provided that $\Theta$ is subdifferentially regular at $a$ satisfying a qualification condition. For the case of local maximizer, we prove that $a$ operator commutes with each element of the (Fenchel) subdifferential of $\Theta$ at $a$ whenever this subdifferential is nonempty. As an application, we characterize local optimizers of shifted strictly convex spectral functions and norms over automorphism invariant sets.
    
    \bigskip
    
    {\it Mathematics Subject Classification}:  17C99, 17C30, 49J52, 90C56, 90C99 
    
    {\it Keywords}:  Euclidean Jordan algebra, automorphism invariance, commutation principle, nonsmooth analysis, strictly convex norm.
}\end{quote}
\bigskip
\bigskip

\section{Introduction}

The commutation principle in Euclidean Jordan algebras established by Ram\'irez, Seeger and Sossa \cite{RSS} has been shown to be a useful tool for analyzing several variational problems where spectral sets and spectral functions are involved (see Section\,\ref{preliminaries} for definitions and main properties concerning the elements of Euclidean Jordan algebras). This principle reads as follows.
\begin{theorem}\label{th:first}
    Let $\mathcal V$ be a Euclidean Jordan algebra and let $\Theta : \VV \to \mathbb{R}$ be a Fr{\'e}chet differentiable function. Let $F : \VV \to \mathbb{R}$ be a spectral function and let $\mathcal{E} \subseteq \VV$ be a spectral set. If $a$ is a local minimizer/maximizer of the map
    \begin{equation}\label{map0}
     x \in \mathcal{E} \mapsto \Theta(x) + F(x), 
     \end{equation}
    then $a$ and $\nabla \Theta(a)$ operator commute.
\end{theorem}

In \cite{GJ2017}, highlighting the role of algebra automorphisms, Gowda and Jeong presented a different proof of the above result in deducing the operator commutative property. Furthermore, they showed that Theorem\,\ref{th:first} still holds if $\EE$ and $F$ are assumed to be weakly spectral (i.e., automorphism invariant) rather than (the stronger assumption of being) spectral.

An example where the commutation principle (the Gowda and Jeong's version) was successfully applied is, for instance, the work of Lourenço and Takeda \cite{Lourenco}, where they obtained formulae for the generalized subdifferentials of spectral functions. Recently, Sossa \cite{Sossa} also used this principle to obtain bounds for the determinant of the sum of two elements in terms of the product of the sum of their eigenvalues, known as a Fiedler-type determinantal inequality. 

Some other commutation principles analogous to Theorem\,\ref{th:first} were stated for other systems such as normal decomposition systems \cite{GJ2017,Niezgoda} and framed Euclidean spaces \cite{Seeger2013}. The commutation principle also inspired a general system called the Fan-Theobald-von Neumann system \cite{Gowda2019,Gowda2022,GJ2023}, which provides a framework to study in a unified manner where a notion of eigenvalue is defined.

The purpose of this work is to present a commutation principle analogous to Theorem\,\ref{th:first} but with $\Theta$ allowed to be nonsmooth. Nonsmooth functions appear in a wide class of optimization problems; see, for instance, the book of Bagirov et al. \cite{Bagirov} and references therein. In particular, functions involving the $k$th largest eigenvalue of a matrix (or of an element of a Euclidean Jordan algebra) are frequently nonsmooth. The following examples formulated in a Euclidean Jordan algebra $\VV$ illustrate this fact. 

\begin{example}\label{ex:kappa}
    Let $u$ be a nonzero element of $\VV$  such that it has nonnegative eigenvalues which we denote by $u\succeq 0$. The condition number of $u$ is defined as
    \[ \kappa(u) := \frac{\lambda_{\max}(u)}{\lambda_{\min}(u)}, \]
    where $\lambda_{\max}(u)$ and $\lambda_{\min}(u)$ denote the largest and smallest eigenvalue of $u$, respectively. In \cite{Seeger2022}, Seeger has studied the problem of perturbing a nonzero element $a\succeq 0$ in such a way as to reduce its condition number as much as possible. This problem can be formulated as follows.
    \begin{equation}\label{kappa}
    \left\{\begin{array}{ll}
         {\rm Minimize}& \kappa(x+a) \\
        \text{subject to} & x\in\EE,  
    \end{array}\right.
    \end{equation}
    where $\EE$ is some set of admissible perturbations, which is usually a spectral set. The objective function of \eqref{kappa} is nonconvex and nonsmooth. So, it fits in the class of problems that we are interested. 
\end{example}

\begin{example}
Let $F(\cdot)$ be a strictly convex spectral norm in $\VV$ (see Section\,\ref{applications} for definitions). For $b\in\VV$, let $\OO_b:=\{Xb:X\in\AutV\}$ be the automorphism orbit of $b$ with $\AutV$ being the set of automorphism in $\VV$. One can see that $\OO_b$ is a weakly spectral set. For $a\in\VV$ consider the problem
\begin{equation}\label{ex:norm}
    \left\{\begin{array}{ll}
         {\rm Minimize/Maximize}&  F (x-a) \\
       \text{subject to}& x\in\OO_b.  
    \end{array}\right.
    \end{equation}
    We observe that the objective function of \eqref{ex:norm} is convex but nonsmooth in general. Thus, \eqref{ex:norm} is also in the class of problems that this work is concerned. We point out that the minimization case of problem \eqref{ex:norm} was studied by Massey, Rios, and Stojanoff \cite{Massey} in the case of the space of Hermitian matrices which is a particular instance of a simple Euclidean Jordan algebra. These authors proved (for the simple algebra of Hermitian matrices) that any local solution (of the minimization case) of \eqref{ex:norm} strongly operator commute (see the definition below) with $a$, and therefore, they deduced that any local solution of \eqref{ex:norm} is a global solution.
\end{example}

Observe that problems \eqref{kappa} and \eqref{ex:norm} have a common structure since they consist of optimizing a shifted spectral function over a weakly spectral set. That is, they belong to the following class of problems:
   \begin{equation}\label{shifted}
    \left\{\begin{array}{ll}
         {\rm Minimize/Maximize}&  F (x-a) \\
       \text{subject to}& x\in\EE,  
    \end{array}\right.
    \end{equation}
where $a\in\VV$, $F:\VV\to\mathbb{R}$ is spectral, and $\EE$ is weakly spectral.

Recently, Gowda \cite{G2022} obtained a strong commutation principle for the maximization case of Theorem\,\ref{th:first} by allowing $\Theta$ to be nonsmooth but admitting a nonempty (restricted) subdifferential. More precisely, for $\Theta:\VV\to\mathbb{R}$ and $S\subseteq \VV$, the \emph{(Fenchel) subdifferential} of $\Theta$ at $a$ relative to $S$ is defined as
\begin{equation}\label{Ssubdif}
    \partial_S\Theta(a) := \set{v \in \VV}{\Theta(x) - \Theta(a) - \ip{v}{x-a} \geq 0 \;\text{ for all }\;x\in S}.
\end{equation}
Gowda's result says that under the assumptions on $\EE$ being a spectral set and $F$ being a spectral function, if $a$ is a global maximizer of the map \eqref{map0} and $\partial_\EE\Theta(a)\neq\emptyset$, then $a$ strongly operator commutes with every element in $\partial_\EE\Theta(a)$. That $a$ and $b$ strongly operator commute means they not only operator commute but also satisfy $\ip{a}{b} = \ip{\lambda(a)}{\lambda(b)}$, where $\lambda(\cdot)$ denotes the eigenvalue map. This remarkable result can be applied to many variational problems, especially when $\Theta$ turns out to be convex. However, the assumption that $\EE$ is spectral rather than weakly spectral is crucial in the Gowda's proof. Hence, this result can not be applied to problem \eqref{ex:norm} since in general $\mathcal O_b$ is weakly spectral but not spectral. Furthermore, if $a$ is a  minimizer of the map \eqref{map0}, the Gowda's technique cannot be used directly to deduce the (strong) operator commutativity between $a$ and some elements of $\partial_\EE\Theta(a)$ unless we have convexity assumptions over $\EE$ and $\Theta$ and by imposing $F=0$. Thus, problem \eqref{kappa} is not covered by the results presented in \cite{G2022}.

The main contributions of this article can be summarized as follows:
\begin{itemize}
    \item[-] In Theorem\,\ref{th:mainmax}, we improve Gowda's commutation principle by allowing the set $\EE$ and the function $F$ to be weakly spectral rather than only spectral. Furthermore, our result also covers the case that $a$ is a local maximizer. The price we pay is that we obtain just the operator commutativity of $a$ with the elements of $\partial_\EE\Theta(a)$ instead of the strong operator commutativity.

    \item[-] In Theorem\,\ref{th:main}, we establish a nonsmooth commutation principle for the minimization case. For instance, by assuming that $\Theta$ is Clarke regular (see Definition\,\ref{Clarke reg}), we show that if $a$ is a local minimizer of the map $x \in \mathcal{E} \mapsto \Theta(x) + F(x)$, then $a$ operator commutes with some element in the Clarke subdifferential of $\Theta$ at $a$. See Corollary\,\ref{cor:mainclarke}.  
    
    \item[-] As an application of our commutation principles, in Theorem\,\ref{th:comaAb} we prove that any local solution of \eqref{shifted} operator commutes with $a$ provided that $F$ is either a strictly convex function or a strictly convex norm.
\end{itemize}

The organization of the paper is as follows: In Section\,\ref{preliminaries} we describe the tools of nonsmooth analysis that we require in this work, and we provide the basic elements and results of Euclidean Jordan algebras. We also analyze in detail the limiting normal cone of a smooth manifold which is a requisite to link the optimality conditions of our optimization problems with the commutation properties of the involved elements. Nonsmooth commutation principles are stated and studied in Section\,\ref{sec:main}. We also discuss some special cases and we apply our results to the condition number minimization problem \eqref{kappa}. Finally, Section\,\ref{applications} is devoted to the analysis of the local solutions of the model \eqref{shifted} by means of our commutation principles provided that $F$ is either a strictly convex spectral function or a strictly convex spectral norm.

\section{Notations and mathematical backgrounds}\label{preliminaries}

Throughout this work, $\ip{\cdot}{\cdot}$ refers to an inner product of each finite dimensional real inner product space that we are considering, and $\norm{\,\cdot\,}$ is the induced norm. In particular, for a finite dimensional real inner product space $\VV$, the symbol $\LV$ stands for the vector space of linear maps from $\VV$ to $\VV$, and we assume that this space is equipped with the Frobenius inner product
\[ \ip{X}{Y} = \sum_{i=1}^m \sum_{j=1}^m \ip{c_i}{Xc_j}\ip{c_i}{Yc_j}, \]
where $\{c_1,\ldots,c_m\}$ is a fixed orthonormal basis of $\VV$. For $u,v\in\VV$, the tensor product $u \otimes v \in \LV$ is given by $(u \otimes v)x := \ip{v}{x} u$ for all $x\in\VV$. For $a\in\VV$, the adjoint of the linear map $\mathcal{A} : \LV \to \VV$ given by $\mathcal{A} (X)=Xa$ is
\begin{equation}\label{adjoint}
    \mathcal{A}^\ast (x)=x\otimes a\;\text{ for all }x\in\VV.
\end{equation}
Hence, $\ip{Xa}{x} = \ip{X}{x \otimes a}$ for all $x, a \in \VV$ and $X \in \LV$.

\subsection{Tools from nonsmooth variational analysis}\label{nosmooth}
In this section we review some elements of nonsmooth variational analysis. We refer to the books by Mordukhovich \cite{Mor2006a, Mor2006b, Mor2018}, and by Rockafellar and Wets \cite{Rock-Wets} where the reader can find full theory, vast bibliography, and commentaries on the subject.

Let $\VV$ be a finite dimensional real inner product space and let $f:\VV\to \bR := [-\infty, +\infty]$ be an extended real-valued function on $\VV$. 
Let $\bx \in \VV$ be such that $f(\bx)$ is finite. The following definitions of generalized subdifferentials are taken from \cite[Definition\,8.3]{Rock-Wets}. The \emph{Fr{\'e}chet subdifferential} of $f$ at $\bx$ is defined by
\begin{equation*}
    \wpartial f(\bx) := \set{v \in \mathcal V}{\liminf_{x \to \bx} \frac{f(x)-f(\bx) - \ip{v}{x-\bx}}{\Vert x-\bx \Vert} \geq 0}.
\end{equation*}
The \emph{limiting (Mordukhovich) subdifferential} of $f$ at $\bx$, written as $\partial f(\bx)$, is the set whose elements are vectors $v\in\VV$ for which there exist sequences $\{x^k\}$ and $\{v^k\}$ on $\VV$ such that 
\[ x^k \to \bx, \quad f(x^k) \to f(\bx), \quad v^k \in \wpartial f(x^k), \quad v^k\to v. \]
The \emph{horizon subdifferential} of $f$ at $\bx$, written as $\partial^{\infty}f(\bx)$, consists of all vectors $v\in \VV$ such that there exist sequences $\{x^k\}$, $\{v^k\}$ in $\VV$, and $\{t_k\}\subset \mathbb{R}_{+}$, where $\mathbb{R}_{+}$ denotes the set of positive real numbers, such that 
\[ x^k \to \bx, \quad f(x^k) \to f(\bx), \quad v^k \in \wpartial f(x^k), \quad t_k \downarrow 0, \quad t_kv^k \to v. \]

Let $C \subset \VV$ be a closed set and let $\bx\in C$. The \emph{(Bouligand) tangent cone} of $C$ at $\bx$, denoted by $\widehat{T}(\bx;\, C)$, consists of all vectors $v\in\VV$ for which there are sequences $\{x^k\} \subset C$ and $\{t^k\}\subset\mathbb{R}_+$ such that (cf. \cite[Definition\,6.1]{Rock-Wets})
\[ x^k \to \bx, \quad t_k \downarrow 0, \quad \frac{x^k-\bx}{t_k} \to v. \]
The \emph{Fr{\'e}chet normal cone} and the \emph{limiting normal cone} to $C$ at $\bx$ are respectively given by (cf. \cite[Definition\,6.3]{Rock-Wets})
\[ \widehat{N} (\bx;\, C) := \wpartial  \delta(\bx;\, C) \quad \text{and} \quad N(\bx;\, C) := \partial \delta(\bx;\, C), \]
where $\delta(\cdot;\, C)$ is the indicator function of $C$. It is known that $\widehat{N} (\bx;\, C)$ coincides with the polar cone of $\widehat{T}(\bx;\, C)$. That is,
\begin{equation}\label{tanchar}
    \widehat{N} (\bx;\, C) = \set{v \in \VV}{\ip{v}{w} \leq 0, \text{ for all } w \in \widehat{T}(\bx;\, C)}.
\end{equation}
The elements of $ N(\bx;\, C)$ are characterized as the vectors $v$ for which there exist sequences $\{x^k\} \subset C$ and $\{v^k\} \subset \VV$ such that $v^k \in \widehat{N} (x^k;\, C)$, $x^k\to\bx$, and $v^k\to v$. Therefore, $ \widehat{N} (\bx;\, C)\subseteq N(\bx;\, C).$

Next, we explain how the chain rule is obtained for these generalized subdifferentials. The chain rule is applicable under a regularity assumption, which is defined below, cf. \cite[Corollary\,8.10]{Rock-Wets}. The \emph{horizon cone} of a nonempty set $C\subseteq\VV$ is the closed cone $C^\infty$ consisting on vectors $x\in\VV$ for which there are sequences $\{x^k\}\subset C$ and $\{t_k\}\subset\mathbb{R}_+$ such that $t_k\downarrow 0$ and $t_k x^k\to x$ (cf. \cite[Definition\,3.3]{Rock-Wets}). 
\begin{definition}\label{def:reg}
A function $f:\VV\to\bR$ is \emph{(subdifferentially) regular} at $\bx\in\VV$ if $f(\bx)$ is finite, $f$ is lower semicontinuous around $\bx$, and
\[ \partial f(\bx) = \wpartial  f(\bx) \neq \emptyset, \quad \partial^\infty f(\bx) = \left(\wpartial  f(\bx)\right)^\infty. \]
\end{definition}
The above definition is obtained from (cf. \cite[Corollary\,8.11]{Rock-Wets}) with the following slight modification: for simplicity in the presentation, we are assuming that the condition $\partial f(\bx)\neq\emptyset$ is part of the definition of regularity.

We quote the simplest version of the chain rule which is enough for our purpose, cf. \cite[Theorem\,10.6]{Rock-Wets}.
\begin{proposition}\label{pr:chain}
    Let $\VV$ and $\WW$ be finite dimensional real inner product spaces. Let $f : \VV \to \bR$ be a function, $A:\WW\to\VV$ be a linear map, and $b\in\VV$. Define $h : \WW \to \bR$ by $h(y) = f(A y+b)$. For $\bar y\in \WW$, suppose that
    \begin{equation}\label{cond:chain}
        \text{$f$ is regular at $A\bar y+b$ \quad and \quad
         ${\rm Ker}(A^\ast)\cap \partial^\infty f(A\bar y+b)=\{0\}$.}
    \end{equation}
    Then, $h$ is regular at $\bar y$ and
    \[ \partial h(\bar y)=A^\ast\partial f(A\bar y+b) \quad \text{and} \quad \partial^\infty h(\bar y)=A^\ast\partial^\infty f(A\bar y+b). \]
\end{proposition}

The next result is a key tool for the proof of our Theorem\,\ref{th:main}. It provides necessary optimality conditions for nonsmooth optimization. Its proof is obtained by applying the Fermat's rule \cite[Proposition\,1.30]{Mor2018} and the subdifferential sum rule \cite[Theorem\,2.19]{Mor2018}.
\begin{theorem}\label{th:opt}
    Let $f:\VV\to\bR$ be a function and let $C$ be a closed subset of $\VV$. Let $\bx\in\VV$ be a local minimizer of the map $x\in C \mapsto f(x)$ and suppose that $f$ is lower semicontinuous around $\bx\in\VV$. Assume that the following qualification condition holds:
    \begin{equation}\label{qconstraint}                                               
        \partial^\infty f(\bx)\cap\left(-N(\bx;C)\right)=\{0\}.
    \end{equation}
    Then, 
    \begin{equation}\label{fermat}
        0\in \partial f(\bx)+N(\bx;C).
    \end{equation}
\end{theorem}

Let us discuss the special case when $f : \VV \to \bR$ is locally Lipschitz around $\bx \in \VV$. The \emph{Clarke generalized directional derivative} of $f$ at $\bx$ in the direction $d \in \VV$ is defined as
\[ f^\circ(\bx; d) = \limsup_{x \to \bx,\; t \,\downarrow\, 0} \frac{f(x+td) - f(x)}{t}. \]
The \emph{Clarke subdifferential} of $f$ at $\bx$ is given by
\[ \partial^\circ f(\bx) := \set{v \in \VV}{f^\circ(\bx; d) \geq \ip{v}{d},\; \forall d \in \VV}. \]
It is known that (cf. \cite[Theorem\,8.6]{Rock-Wets})
\begin{equation}\label{incl}
    \wpartial f(\bx)\subseteq\partial f(\bx)\subseteq \partial^{\circ} f(\bx).
\end{equation}
\begin{definition} \label{Clarke reg}
    (\cite[Definition\,2.3.4]{Clarke}) We say that $f: \VV \to \mathbb{R}$ is \emph{Clarke regular} at $\bx\in\Rn$  if $f$ is locally Lipschitz around $\bx$ and if for every $d\in\Rn$, the ordinary directional derivative
    \[ f^\prime(\bx;d) := \lim_{t \,\downarrow\, 0} \frac{f(\bx+td)-f(\bx)}{t} \]
    exists and coincides with the generalized one, i.e., $f^\prime(\bx;d) = f^\circ(\bx;d)$.
\end{definition}
Since $\wpartial f(\bx) = \set{v \in \VV}{f^\prime(\bx;d) \geq \ip{v}{d},\;\forall d \in \VV}$ (cf. \cite[Chapter\,8]{Rock-Wets}), when $f$ is Clarke regular at $\bx$ we have that all the subdifferentials in \eqref{incl} coincide; i.e., $\wpartial f(\bx)=\partial f(\bx)= \partial^{\circ} f(\bx).$ From \cite[Theorem\,8.6]{Rock-Wets} we know that $ \big( \wpartial f(\bx) \big)^\infty \subseteq \partial^\infty f(\bx)$. Furthermore, because $f$ is locally Lipschitz around $\bx$ we have that $\partial^\infty f(\bx)=\{0\}$. Thus, when $f$ is Clarke regular at $\bx$ we have that $f$ is regular at $x$ (cf. Definition\,\ref{def:reg}), and the qualification condition \eqref{qconstraint} is automatically satisfied.

\subsection{Euclidean Jordan algebras}
Let $\VV$ be a finite dimensional real inner product space. If $\VV$ is equipped with a Jordan product $\circ : \VV \times \VV \to \VV$ such that it is (i) commutative, (ii) satisfies the Jordan identity,
\[ x^2 \circ (x \circ y) = x \circ (x^2 \circ y), \]
where $x^2 = x \circ x$, and is (iii) comparable with the underlying inner product, i.e.,
\[ \ip{x \circ y}{z} = \ip{y}{x \circ z}, \]
for $x, y, z \in \VV$, then the triple $(\VV, \circ, \ip{\cdot}{\cdot})$, or simply $\VV$, is called a \emph{Euclidean Jordan algebra}. An element $e \in \VV$ is a unit if $x \circ e = x$ for all $x \in \VV$. Throughout the paper, we assume the existence of the unit element. It is well known \cite[Corollary IV.1.5, Theorem V.3.7]{FK} that any Euclidean Jordan algebra can be decomposed into a direct product of simple Euclidean Jordan algebras and there are only five simple ones up to isomorphism. The space of $n \times n$ Hermitian matrices together with the Jordan product and inner product defined respectively by
\[ X \circ Y = \frac{XY + YX}{2}, \quad \ip{X}{Y} = \operatorname{tr}(XY)\;\mbox{(trace inner product),} \]
is a primary example of a simple Euclidean Jordan algebra.

The spectral decomposition theorem \cite[Theorem III.1.2]{FK} asserts that there exists a positive integer $n$, which we will call the rank of $\VV$, such that any element $x \in \VV$ can be written as
\[ x = \lambda_1(x) e_1 + \lambda_2(x) e_2 + \cdots + \lambda_n(x) e_n, \]
where $\lambda_1(x) \geq \lambda_2(x) \geq \cdots \geq \lambda_n(x)$ are real numbers called the eigenvalues of $x$, and $\{e_1, e_2, \ldots, e_n\}$ is a Jordan frame. We define the \emph{eigenvalue map} $\lambda : \VV \to \Rn$ as follows:
\[ \lambda(x) = \big( \lambda_1(x), \lambda_2(x), \ldots, \lambda_n(x) \big) \in \Rn. \]
Thus, $\lambda(x)$ is the vector consisting of the eigenvalues of $x$ written in nonincreasing order. It is known \cite{Baes} that $\lambda : \VV \to \Rn$ is (Lipschitz) continuous. 

For $a \in \VV$, define a linear map $L_a : \VV \to \VV$ by $L_a x = a \circ x$ for all $x \in \VV$, which we call the \emph{Lyapunov map}. Then, an easy calculation verifies that
\begin{equation} \label{Lyapunov}
    \ip{a}{(L_uL_v-L_vL_u)b} = \ip{u}{(L_aL_b-L_bL_a)v},
\end{equation}
for all $a, b, u, v \in \VV$. We say that elements $a, b \in \VV$ \emph{operator commute} provided the corresponding Lyapunov maps commute, which means that $L_a(L_b x) = L_b(L_a x)$ for all $x \in \VV$. In the case of Hermitian matrices, the operator commutativity is equivalent to the usual (matrix) commutativity.

It is known \cite{FK} that $a, b \in \VV$ operator commute if and only if both $a$ and $b$ can be simultaneously decomposed by the same Jordan frame. That is, there exists a Jordan frame $\{e_1, e_2, \ldots, e_n\}$ and real numbers $\alpha_i, \beta_i \in \mathbb{R}$ for $i = 1, 2, \ldots, n$ such that
\[ a = \alpha_1 e_1 + \alpha_2 e_2 + \cdots + \alpha_n e_n, \quad b = \beta_1 e_1 + \beta_2 e_2 + \cdots + \beta_n e_n, \]
and $\alpha$ and $\beta$ are some rearrangement of $\lambda(a)$ and $\lambda(b)$, respectively. We say that $a, b \in \VV$ \emph{strongly operator commute} if they operator commute with the additional condition that $\alpha_i = \lambda_i(a)$ and $\beta_i = \lambda_i(b)$ for each $i = 1, 2, \ldots, n$. Then $a$ and $b$ strongly operator commute if and only if $\ip{a}{b} = \ip{\lambda(a)}{\lambda(b)}$, which is equivalent to $\lambda(a+b)=\lambda(a)+\lambda(b)$ (cf. \cite[Proposition 2.6]{Gowda2019}).

A linear map $X : \VV \to \VV$ is an \emph{(algebra) automorphism} if $X$ is invertible and
\[ X(a \circ b) = Xa \circ Xb, \]
for all $a, b \in \VV$. The set of all automorphisms of $\VV$ is denoted by $\AutV$. As an example, an algebra automorphism of the space of Hermitian matrices is either of the form $X \mapsto U^{\ast}XU$ or $X \mapsto U^{\ast} \overline{X}U$, where $\overline{X}$ is a complex conjugate of $X$ and $U$ is a unitary matrix (cf. \cite[Theorem\,10]{Or24}).

A linear map $D : \VV \to \VV$ is a \emph{derivation} if $D$ satisfies the Leibniz's rule,
\[ D(a \circ b) = Da \circ b + a \circ Db, \]
for all $a, b, \in \VV$. The set of all derivations of $\VV$ is denoted by $\DerV$.

In the next proposition, we present an alternative characterization of the operator commutativity in terms of derivations of $\VV$.
\begin{proposition}\label{comder}
    Let $\VV$ be a Euclidean Jordan algebra. For $a, b \in \VV$, the following statements are equivalent:
    \begin{itemize}
      \item[$(a)$] $a$ and $b$ operator commute.
      \item[$(b)$] $\ip{a}{Db} = 0$ for all $D \in \DerV$.
      \item[$(c)$] $a\otimes b\in{\rm Der}(\VV)^\perp$.
    \end{itemize}
\end{proposition}
%

\begin{proof}
Let us prove $(a)\Leftrightarrow(b)$.
    Suppose that $a$ and $b$ operator commute so that $L_aL_b = L_bL_a$. Let $D \in \DerV$. Since every Euclidean Jordan algebra is semisimple, from \cite{Jacobson1949} (cf. \cite[Chapter IV, Theorem\,8]{Koecher}), we know that
    \[ D = \sum_{i=1}^{k} (L_{u_i}L_{v_i} - L_{v_i}L_{u_i}), \] 
    for some positive integer $k$ and some elements $u_i, v_i \in \VV$, $i = 1, 2, \ldots, k$. Hence, thanks to \eqref{Lyapunov}, we deduce that
    \[ \ip{a}{Db} = \sum_{i=1}^k \ip{a}{(L_{u_i}L_{v_i} - L_{v_i}L_{u_i})b}
        = \sum_{i=1}^k \ip{u_i}{(L_aL_b - L_bL_a)v_i} = 0. \]
    Hence, $(b)$ is satisfied. The converse of the statement is proved in \cite{GJ2017} and is given here for the sake of completeness. Assume that $(b)$ is true. Then, since $D: = L_uL_v - L_uL_v$ for any $u, v \in \VV$ is a derivation, we see that from $(b)$ $\ip{u}{(L_aL_b - L_bL_a)v} = 0$. As $u, v \in \VV$ are arbitrary, it follows that $L_aL_b - L_bL_a = 0$. Thus, $a$ and $b$ operator commute. The equivalence between $(b)$ and $(c)$ is a consequence of the equality $\ip{a}{Db} = \ip{a \otimes b}{D}$ which is explained in \eqref{adjoint}.
\end{proof}

Recall that a set $Q \in \Rn$ is said to be \emph{symmetric} if the implication $u \in Q \Rightarrow Pu \in Q$ holds for all $n \times n$ permutation matrices $P$. Also, a function $f : \Rn \to \mathbb{R}$ is said to be \emph{symmetric} if $f(Pu) = f(u)$ for all $n \times n$ permutation matrices $P$. Extending these to Euclidean Jordan algebras, a set $E \subset \VV$ is called a \emph{spectral set} if there exists a symmetric set $Q \subset \Rn$ such that $E = \lambda^{-1}(Q)$, and $E$ is called a \emph{weakly spectral set} if $E$ is invariant under automorphisms, that is, the implication $x \in E \Rightarrow X a \in E$ holds for all $X \in \AutV$. Likewise, a function $F : \VV \to \mathbb{R}$ is said to be a \emph{spectral function} if there exists a symmetric function $f : \Rn \to \mathbb{R}$ such that $F(x) = f(\lambda(x))$ for all $x \in \VV$, and $F$ is a \emph{weakly spectral function} if $F(Xa) = F(a)$ holds for all $a \in \VV$ and $X \in \AutV$. It has been shown \cite{JG2017} that all spectral sets/functions are weakly spectral, but the converse holds only when $\VV$ is simple or $\Rn$. In particular, in the space of Hermitian matrices (which is simple), spectrality and weak spectrality are equivalent.

\subsection{The limiting normal cone of a smooth manifold}
As we see in Theorem\,\ref{th:opt}, to extract information from the optimality condition \eqref{fermat}, we are required to describe the limiting normal cone of the feasible set of the optimization problem of our interest. We will see in the next section that the study of our optimization problems reduces to considering some optimization problems formulated over the feasible set $\AutV$. This set has the remarkable property of being a smooth manifold, which can be explained by the fact that $\AutV$ has a Lie group structure. Recall that a \emph{Lie group} in a finite dimensional real inner product space $\VV$ is a nonempty subset $\GG$ of $\VV$ such that it has a group structure, it is a smooth manifold in $\VV$, and the group operation $(g_1, g_2) \in \GG \times \GG \mapsto g_1 \cdot g_2$ and the inverse map $g \in \GG \mapsto g^{-1}$ are both smooth. The Cartan's closed subgroup theorem (cf. \cite[Theorem\,20.12]{Lee}) asserts that any closed subgroup of a Lie group $\GG$ is a Lie subgroup (and thus a smooth submanifold) of $\GG$. Note that the space $\GL(\VV)$ of invertible linear transformations from $\VV$ to itself is a Lie group with the composition as a group operation. This Lie group is known as the \emph{general linear group} of $\VV$. Since $\AutV$ is a closed subgroup of $\GL(\VV)$ (cf. \cite[Section\,II.4]{FK}), Cartan's closed-subgroup theorem deduces that $\AutV$ is a Lie (sub)group of $\GL(\VV)$. Therefore, it is a smooth manifold in $\LV$.

Let $\MM$ be a smooth manifold embedded in a finite dimensional real inner product space $\VV$. For every point $p\in\MM$, the \emph{tangent space} $T(p;\MM)$ of $\MM$ at $p$ is a linear subspace of $\VV$ consisting of vectors $h\in\VV$ for which there exists a smooth curve $\gamma : (-\epsilon, \epsilon) \to \MM$ with $\epsilon>0$, such that
\[ \gamma(0) = p \quad \text{and} \quad \gamma^\prime(0) = h, \]
(cf. \cite[Definition\ 4.4]{Gallier}). The \emph{normal space} of $\MM$ at $p$ is $T(p;\MM)^\perp$; that is, the orthogonal complement to the tangent space $T(p;\MM)$.

It is mentioned in the literature (e.g., \cite[Example\,6.8]{Rock-Wets}), in the case of a smooth manifold of a finite dimensional real inner product space $\VV$, the Bouligand tangent cone and the limiting normal cone reduce to the tangent and normal spaces, respectively, associated with such a manifold in differential geometry. However, this affirmation is left as an exercise in the book of Rockafellar and Wets \cite{Rock-Wets} and it seems not easy to find a suitable reference where the proof is given. For this reason, we state the following proposition with our terminology and provide a detailed proof in the appendix.

\begin{proposition}\label{normmanifold}
    Let $\MM$ be a smooth manifold in a finite dimensional real inner product space $\VV$ and let $p\in\MM$. Then,
    \begin{itemize}
        \item[$(a)$] $T(p;\MM)=\widehat{T}(p;\MM),$\label{one}
        \item[$(b)$] $\widehat{N}(p;\MM)=T(p;\MM)^\perp,$\label{two}
        \item[$(c)$] $N(p;\MM)=\widehat{N}(p;\MM)$.\label{three}
    \end{itemize}
\end{proposition}

Since $\DerV$ is the Lie algebra of a Lie group $\AutV$ (cf. \cite[Section\,II.4]{FK}), we have
\[ T(I; \AutV) = \DerV, \]
where $I$ is the identity map on $\VV$. Hence, thanks to Proposition\,\ref{normmanifold}, we have the following.

\begin{proposition}\label{tangentnormal}
    Let $\VV$ be a Euclidean Jordan algebra and let $I$ be the identity map on $\VV$. Then,
    \[ N(I;\AutV) = \DerV^\perp. \]
\end{proposition}

\section{Nonsmooth commutation principles}\label{sec:main}

We present our nonsmooth commutation principle for the maximization case. Recall that $\partial_\EE\Theta(a)$ denotes the subdifferential of $\Theta$ at $a$ relative to $\EE$, defined as in \eqref{Ssubdif}. 

\begin{theorem} \label{th:mainmax}
    Let $\mathcal V$ be a Euclidean Jordan algebra and let $\Theta : \VV \to \bR$ be a function. Let $F : \VV \to \mathbb{R}$ be a weakly spectral function and let $\mathcal E\subseteq \VV$ be a weakly spectral set. Let $a \in \VV$ be such that $\Theta(a)$ is finite and $\partial_\EE\Theta(a) \neq \emptyset$. Suppose that $a$ is a local maximizer of the map
    \begin{equation} \label{mapmax}
        x\in \mathcal E \mapsto \Theta(x)+F(x).
    \end{equation}
    Then, $a$ operator commutes with each element of $\partial_\EE\Theta(a)$.
\end{theorem}
\begin{proof}
    Let $a$ be a local maximizer of \eqref{mapmax}. Then,
    \begin{equation}\label{ineqmax1}
    \Theta(a)+F(a)\;\geq\; \Theta(x)+F(x),\;\text{for all }x\in\mathcal E\cap  N_a,
    \end{equation}
    for some $ N_a$ neighborhood of $a$. Since $\Vert Xa-a\Vert\leq \Vert X-I\Vert \Vert a\Vert$ for all $X \in \LV$, there exists a neighborhood $\mathcal N_I$ of $I$ in $\LV$ such that $Xa\in N_a$ for all $X\in\mathcal N_I$. Furthermore, $Xa\in\mathcal E$ for all $X\in\AutV$ because $\mathcal E$ is weakly spectral. Thus, $Xa\in\mathcal E\cap  N_a$ for all $X\in \AutV\cap \mathcal N_I$. Then, from \eqref{ineqmax1} we deduce that
    \begin{equation}\label{ineq2a}
    \Theta(a)+F(a)\;\geq\; \Theta(Xa)+F(Xa),\;\text{for all }X\in\AutV\cap \mathcal N_I.
    \end{equation}
    Now, $F(a)=F(Xa)$ for all $X\in\AutV$ because $F$ is weakly spectral. Hence, \eqref{ineq2a} becomes
    \begin{equation}\label{ineqmax2}
    \Theta(a)\;\geq\; \Theta(Xa),\;\text{for all }X\in\AutV\cap \mathcal N_I.
    \end{equation}
    Let $v\in \partial_\EE\Theta(a)$. From \eqref{Ssubdif} we have that
    \begin{equation}\label{ineqba}
    \Theta(x) - \Theta(a) - \ip{v}{x-a} \geq 0,
    \end{equation}
    for all $x\in\EE$. By combining \eqref{ineqmax2} and \eqref{ineqba} we deduce that
    \[ \ip{v}{a} \geq \ip{v}{Xa}, \; \text{for all} \; X \in \AutV \cap \mathcal N_I. \]
    It means that $I$ is a local maximizer of the map $X \in \AutV \mapsto \ip{v}{Xa}$. From \eqref{adjoint} we have that $\ip{v}{Xa} = \ip{v \otimes a}{X}$. Thus, $I$ is a local solution of
    \begin{equation}\label{maxdif}
        \operatorname*{maximize}\limits_{X \in \AutV} \;\; \ip{X}{v \otimes a}.
    \end{equation}
    Since the objective function of \eqref{maxdif} is differentiable, the first-order condition for optimality says  $-v\otimes a\in N(I,\AutV)$ (cf. \cite[Theorem\,6.12]{Rock-Wets}). From Proposition\,\ref{tangentnormal} we deduce that $v\otimes a\in\DerV^\perp.$ Thus, from Proposition\,\ref{comder} we conclude that $a$ and $v$ operator commute.
\end{proof}

For the minimization case, we will use the necessary optimality conditions provided in Theorem\,\ref{th:opt}. We will see that the qualification condition \eqref{qconstraint} reduces to the fact that among all the elements in $\partial^\infty\Theta(a)$, the zero vector is the only one that operator commutes with $a$.
\begin{theorem}\label{th:main}
    Let $\mathcal V$ be a Euclidean Jordan algebra, and let $\Theta:\VV\to\bR$ be a function. Let $F:\VV\to\mathbb{R}$ be a weakly spectral function, and let $\mathcal E\subseteq \VV$ be a weakly spectral set. Let $a\in\VV$ be such that $\Theta$ is regular at $a$. Suppose that $a$ is a local minimizer of the map
    \begin{equation}\label{mapT}
     x\in \mathcal E \mapsto \Theta(x)+F(x),
     \end{equation}
    and suppose that the following condition is satisfied:
    \begin{equation}\label{quali}
    \text{If $a$ operator commutes with $v\in\partial^\infty\Theta(a)$, then $v=0$.}
    \end{equation}
    Then, $a$ operator commutes with some element in $\partial\Theta(a)$.
\end{theorem}
\begin{proof}
    Observe that $\partial\Theta(a)\neq\emptyset$ because $\Theta$ is regular at $a$. If $a=0$, then the result follows due to the fact that the zero element operator commutes with every element in $\VV$. Now, assume that $a\neq 0$ and that $a$ is a local minimizer of the map \eqref{mapT}. This means that
    \begin{equation*}
    \Theta(a)+F(a) \leq \Theta(x)+F(x),\;\text{for all }x\in\mathcal E\cap  N_a,
    \end{equation*}
    where $ N_a$ is some neighborhood of $a$. Analogous to the proof of Theorem\,\ref{th:mainmax}, we deduce that 
   \begin{equation*}
    \Theta(a) \leq \Theta(Xa),\;\text{for all }X\in\AutV\cap \mathcal N_I,
    \end{equation*}
    for some $\mathcal N_I$, a neighborhood of $I$ in $\LV$. It means that $I$ is a local solution of
    \begin{equation}\label{opt}
        \operatorname*{minimize}_{X \in \AutV} \;\; \Psi(X):=\Theta(Xa).
    \end{equation}
   To obtain our desired result we will look at the optimality conditions for $I$ to be a local solution of \eqref{opt} provided by Theorem\,\ref{th:opt}. First, observe that $\Psi$ is the composition of $\Theta$ with the linear map $\mathcal{A} : \LV \to \VV$ given by $\mathcal{A}(X)=Xa$ for all $X \in \LV$. Then, to obtain the limiting and horizon subdifferential of $\Psi$ at $I$, we apply the chain rule given in Proposition\,\ref{pr:chain}. To do so, we should verify that \eqref{cond:chain} is satisfied. Indeed, $\Theta$ is regular at $Ia=a$ by the hypothesis. Moreover, from \eqref{adjoint}, we know that $\mathcal{A}^\ast x=x\otimes a$ for all $x\in \VV$. Thus, $x\in{\rm Ker}(\mathcal{A}^\ast)$ if and only if $\ip{a}{v}x = 0$ for all $v\in\VV$. It follows that $x=0$ as $a$ is assumed to be nonzero. We have proved that ${\rm Ker}(\mathcal{A}^\ast)=\{0\}$ and $\eqref{cond:chain}$ is satisfied. Then, Proposition\,\ref{pr:chain} says that $\Psi$ is regular at $I$ and
    \begin{equation}\label{chain}
        \partial\Psi(I)=\{v\otimes a:v\in\partial\Theta(a)\},\quad \partial^\infty\Psi(I)=\{v\otimes a:v\in\partial^\infty\Theta(a)\}.
    \end{equation}
    Recall that from Definition\,\ref{def:reg}, the regularity of $\Psi$ at $I$ implies that $\Psi$ is lower semicontinuous around $I$. Furthermore, as $\VV$ is finite-dimensional, $\AutV$ is closed in $\LV$ equipped with the topology induced by the operator norm. Thus, to apply Theorem\,\ref{th:opt} it remains to verify the qualification condition \eqref{qconstraint}. From Proposition\,\ref{tangentnormal}, we know that $N(I;\AutV) = \DerV^\perp$. If $H \in \partial^\infty\Psi(I)\cap \DerV^\perp$, then we have $H= v\otimes a$ for some $v\in\partial^\infty \Theta(a)$, and $ v\otimes a\in\DerV^\perp$. Hence, we see that $a$ and $v$ operator commute thanks to Proposition\,\ref{comder}. Therefore, condition \eqref{quali} implies that $v=0$; hence, the hypothesis of Theorem\,\ref{th:opt} is satisfied. Then, from \eqref{fermat} we obtain
    \[ 0\in\partial\Psi(I)+N(I;\AutV). \]
    Because of the characterization of $\partial\Psi(I)$ given in \eqref{chain} and the characterization of $N(I;\AutV)$ given in Proposition\,\ref{tangentnormal}, we deduce that there exists $v\in\partial\Theta(a)$ such that $-v\otimes a\in \DerV^\perp$. Then,  from Proposition\,\ref{comder}, we conclude that $a$ and $v$ operator commute.
\end{proof}

When $\Theta:\VV\to\mathbb{R}$ is Clarke regular at $a$ (cf. Definition\,\ref{Clarke reg}), we have that the qualification condition \eqref{quali} holds automatically. Recall also that in this case, $\partial\Theta(a)$ coincides with $\partial^\circ\Theta(a)$, the Clarke subdifferential of $\Theta$ at $a$. Hence, Theorem\,\ref{th:main} becomes:
\begin{corollary}\label{cor:mainclarke}
    Let $\mathcal V$ be a Euclidean Jordan algebra and let $\Theta : \VV \to \mathbb{R}$ be a function. Let $F : \VV \to \mathbb{R}$ be a weakly spectral function and let $\mathcal{E} \subseteq \VV$ be a weakly spectral set. Let $a \in \VV$ be such that $\Theta$ is Clarke regular at $a$. Suppose further that $a$ is a local minimizer of the map $x \in \mathcal{E} \mapsto \Theta(x) + F(x)$. Then, $a$ operator commutes with some element in $\partial^\circ\Theta(a)$. 
\end{corollary}

\begin{example}
    Let us consider the problem 
    \begin{equation}\label{kappa2}
        \left\{ \begin{array}{ll}
        {\rm Minimize} & \Theta(x) := \kappa(x+a) \\
        \text{subject to} & x \in \EE,  
        \end{array} \right.
    \end{equation}
    given in Example\,\ref{ex:kappa}, where $\EE\subset\VV$ is weakly spectral. Let $\bx$ be a local solution of \eqref{kappa2} and suppose that $\bx+a$ has positive eigenvalues (that is, $\bx+a$ is in the interior of the symmetric cone $\{z\circ z:z\in\VV\}$). Seeger proved in \cite[Section\,3]{Seeger2022} that $\kappa$ is Clarke regular at $\bx+a$. Hence, $\Theta$ is also Clarke regular at $\bx$. Thus, by applying the chain rule given in Proposition\,\ref{pr:chain} we get $\partial^\circ\Theta(\bx)=\partial^\circ\kappa(\bx+a)$. Therefore, from Corollary\,\ref{cor:mainclarke}, we deduce that $\bx$ operator commute with some element of $\partial^\circ\kappa(\bx+a)$. The computation of this subdifferential is given in \cite[Theorem\,1]{Seeger2022}.
\end{example}

Now, we specialize Theorems \ref{th:mainmax} and \ref{th:main} to the case when $\Theta$ is convex. We say that a convex function $\Theta : \VV \to \bR$ is \emph{proper} if $\Theta(x)>-\infty$ for all $x\in\VV$ and ${\rm dom}\,\Theta:=\{x\in\VV : \Theta(x)<\infty\}$ is nonempty. For a proper convex function $\Theta$, the subdifferential of $\Theta$ at $x\in{\rm dom}\,\Theta$ coincides with the Fréchet subdifferential and limiting subdifferential of $\Theta$ at $x$ (cf. \cite[Proposition\,1.2]{Kruger}). That is, for all $x\in{\rm dom}\,\Theta$,
\begin{equation}\label{equal}
 \partial \Theta(x) = \wpartial \Theta(x) = \set{v \in \VV}{\Theta(x) - \Theta(a) - \ip{v}{x-a} \geq 0 \; \text{for all} \; x \in \VV}.
 \end{equation}
Let ${\rm ri}\,C$ and ${\rm int}\,C$ denote the \emph{relative interior} and \emph{interior} of the set $C\subset\VV$, respectively. When  $x\in{\rm ri}({\rm dom}\,\Theta)$, $\partial \Theta(x)$ is nonempty (closed and convex). Thus, Theorem\,\ref{th:mainmax} becomes: 
  \begin{corollary}\label{cor:mainmax}
    Let $\Theta:\VV\to\bR$ be a proper convex function. Let $F:\VV\to\mathbb{R}$ be a weakly spectral function, and let $\mathcal E\subseteq \VV$ be a weakly spectral set. Let $a\in{\rm ri}({\rm dom}\,\Theta)$. Suppose that $a$ is a local maximizer of the map
    $
     x\in \mathcal E \mapsto \Theta(x)+F(x).
     $
    Then, $a$ operator commutes with each element of $\partial\Theta(a)$.
\end{corollary}

Suppose that $x\in{\rm int}({\rm dom}\,\Theta)$. It is known that in this case $\Theta$ is locally Lipschitz around $x$, and $\partial^\infty \Theta(x) = \big(\wpartial \Theta(x)\big)^{\infty} = \{0\}$ (cf. \cite[Theorem\,2.29]{MorNam}). Thus, $\Theta$ is regular at $x$. Furthermore, the condition \eqref{quali} is satisfied at $a\in {\rm int}({\rm dom}\,\Theta)$ because $\partial^\infty\Theta(a) = \{0\}$. Thus, Theorem\,\ref{th:main} becomes:
\begin{corollary}\label{cor:main}
    Let $\Theta:\VV\to\bR$ be a proper convex function. Let $F:\VV\to\mathbb{R}$ be a weakly spectral function, and let $\mathcal E\subseteq \VV$ be a weakly spectral set. Let $a\in {\rm int}({\rm dom}\,\Theta)$. Suppose that $a$ is a local minimizer of the map
    $
     x\in \mathcal E \mapsto \Theta(x)+F(x).
     $
    Then, $a$ operator commutes with some element in $\partial\Theta(a)$.
\end{corollary}

\section{Optimization of shifted spectral functions/norms over weakly spectral sets}\label{applications}
%
The purpose of this section is to apply our commutation principles to characterize the local solutions of the problem
    \begin{equation*}\label{minimaxi}
    \left\{\begin{array}{ll}
         {\rm Minimize/Maximize}&  F (x-a) \\
       \text{subject to}& x\in\EE,  
    \end{array}\right.
    \end{equation*}
where $a\in \VV$, $\EE\subset\VV$ is weakly spectral, and $F:\VV\to\mathbb{R}$ is a spectral function that is either a strictly convex function or a strictly convex norm. That $F$ is a \emph{strictly convex norm} in $\VV$ means that it is a norm in $\VV$ in which the unit ball of $\VV$ with respect to this norm is strictly convex. That is, for any distinct two points $x$ and $y$ on the boundary of the unit ball, their midpoint must be strictly inside the unit ball. This is specified in the next definition.

\begin{definition}
    We say that $F : \VV \to \mathbb{R}$ is a \emph{strictly convex norm} in $\VV$ if it is a norm in $\VV$ satisfying $F(x + y) < 2$ for any distinct nonzero $x, y \in \VV$ such that $F(x) = F(y) = 1$.
\end{definition}

\begin{example}
    Let $p\geq 1$. The Schatten $p$-norm of $x\in\VV$ is defined as 
    \[ \Norm{x}_p := \big( \abs{\lambda_1(x)}^p + \cdots + \abs{\lambda_n(x)}^p \big)^\frac{1}{p}. \]
    That is, $\Norm{x}_p = \Vert \lambda(x)\Vert_p$, where $\norm{\,\cdot\,}_p$ is the $p$-norm on $\Rn$. For $p>1$, the Minkowski inequality ensures that $\Norm{\,\cdot\,}_p$ is a strictly convex norm on $\VV$. However, the norm $\Norm{x}_1 = \abs{\lambda_1(x)} + \cdots + \abs{\lambda_n(x)}$ is not strictly convex. Indeed, if $\{e_1, \ldots, e_n\}$ is a Jordan frame of $\VV$, we have that $e_1$ and $e_2$ are nonzero distinct elements of $\VV$ satisfying $\Norm{e_1}_1 = \Norm{e_2}_1 = 1$ but $\Norm{e_1 + e_2}_1 = 2$.
\end{example}

From now on, whenever we say that \emph{$F$ is a strictly convex spectral function/norm in $\VV$}, we refer to that $F$ is spectrally defined and either is strictly convex (in the usual sense) or is a strictly convex norm. As $F$ is spectrally defined, there exists a permutation invariant function $f:\Rn\to\mathbb{R}$ such that $F=f\circ \lambda$. In this case, the strict convexity property of $F$ is transferred to $f$ and vice versa, see Proposition\,\ref{prop:strictconvexity} in the Appendix for its proof. Furthermore, in the Proposition\,\ref{prop:norm} given in the Appendix, it is shown that $F$ is a strict convexity norm on $\VV$ if and only if so is $f$ on $\Rn$.

Let us assume that $F=f\circ\lambda$ is a strictly convex spectral function/norm in $\VV$. Then, $F$ is always convex. Thus, from \eqref{equal} we have that $\partial F(x)$ is the subdifferential of $F$ at $x$. Moreover, its subdifferential is characterized as (cf. \cite{Baes})
\begin{equation}\label{subdifspec}
    \partial F(x)=\{v\in\VV: \lambda(v)\in\partial f(\lambda(x)),\;\text{$v$ and $x$ strongly operator commute}\}.
\end{equation}
The assumptions of strict convexity or norm strict convexity allow us to obtain the following sort of strict monotonicity of the eigenvalues of the elements of $\partial F(x)$ with respect to the eigenvalues of $x$. Its proof is given in the Appendix.
\begin{lemma}\label{monot}
    Let $F$ be a strictly convex spectral function/norm in $\VV$. Let $x\in \VV$ and $v\in\partial F(x)$. Then, for all $i,j \in \{1,2,\ldots,n\}$,
    \[ \lambda_i(x) > \lambda_j(x) \implies \lambda_i(v) > \lambda_j(v). \]
\end{lemma}
Let $v\in\partial F(x)$. From \eqref{subdifspec} we have that $x$ strongly operator commutes with $v$. If, in addition, $v$ operator commutes with some $w\in\VV$, then thanks to Lemma\,\ref{monot} we can deduce that $x$ and $w$ operator commutes as well. This transitive property is stated in the next proposition, whose proof is also given in the Appendix.
\begin{proposition}\label{prop:transitive}
    Let  $a, b, c \in \VV$. Suppose that $a$ and $b$ strongly operator commute and $b$ and $c$ operator commute. If
    \[ \lambda_i(a) > \lambda_j(a) \implies \lambda_i(b) > \lambda_j(b) \]
    for all $i,j \in \{1,2,\ldots,n\}$, then $a$ and $c$ operator commute.
\end{proposition}
Now, we present the main result of this section.
\begin{theorem}\label{th:comaAb}
    Let $a\in\VV$ and let $\EE\subset\VV$ be weakly spectral. Let $F$ be a strictly convex spectral function/norm in $\VV$. Suppose that $\bar{x}$ is a local  minimizer/maximizer of
    \begin{equation}\label{orbfunc}
        x \in \mathcal E \mapsto F(x-a).
    \end{equation}
     Then, $\bar{x}$ and $a$ operator commute.
\end{theorem}
\begin{proof}
    We see that the objective function $\Theta(x):= F(x-a)$ is convex because it is the composition of a convex function $F$ with the affine function $x \mapsto x-a$. Hence, we can use the chain rule given in Proposition\,\ref{pr:chain} to deduce that
    \begin{equation}\label{nablaTF}
        \partial\Theta(x)=\partial F(x-a).
    \end{equation}
    Suppose now that $\bx$ is a local optimizer of \eqref{orbfunc}. If $\bx$ is a local maximizer of \eqref{orbfunc}, from Corollary\,\ref{cor:mainmax} we deduce that $\bx$ operator commutes with each element of $\partial\Theta(\bx)$. On the other hand, if $\bx$ is a local minimizer of \eqref{orbfunc}, from Corollary\,\ref{cor:main} we have that $\bx$ operator commutes with some element in $\partial\Theta(\bx)$. In any case, from \eqref{nablaTF}, we conclude that there exists $v\in\partial F(\bx-a)$ such that 
    \begin{equation}\label{comm1}
        \text{$v$ and $\bx$ operator commute.}
    \end{equation}
    Furthermore, from the characterization \eqref{subdifspec} we deduce that
    \begin{equation}\label{comm2}
    \text{$\bx-a$ and $v$ strongly operator commute.}
    \end{equation}
    Thanks to Lemma\,\ref{monot} and Proposition\,\ref{prop:transitive} applied to \eqref{comm1} and \eqref{comm2}, we obtain that $\bx-a$ and $\bx$ operator commute. Now, for any $D \in \DerV$, we have
    \begin{equation}\label{equalD}
        \ip{a}{D \bx} = \ip{\bx}{D\bx} - \ip{\bx - a}{D\bx}.
    \end{equation}
    Proposition\,\ref{comder} implies that the right expression of \eqref{equalD} vanishes because $\bx$ operator commutes with itself and with $\bx-a$; hence, $\ip{a}{D\bx} = 0$. Since $D \in \DerV$ is arbitrary, we conclude that $\bx$ and $a$ operator commute.
\end{proof}

\section*{Appendix}

\begin{proof}[Proof of Proposition\,\ref{normmanifold}.]
    $(a)$ Let $h\in T(p;\MM)$ so that there exists a smooth curve $\gamma:(-\epsilon,\epsilon)\to \MM$ with $\epsilon>0$, such that $\gamma(0)=p$ and $\gamma^\prime(0)=h$. We take a sequence $\{t_k\}\subset (0,\epsilon)$ such that $t_k\downarrow 0$. By definition of $\gamma$ we have that $\gamma(t_k)\in\MM$ for all $k$, $\gamma(t_k)\to \gamma(0)=p$, and 
    \[ h = \gamma'(0) = \lim_{t\to 0}\frac{\gamma(t)-p}{t} = \lim_{k\to \infty} \frac{\gamma(t_k)-p}{t_k}. \]
    It means that $h\in \widehat{T}(p; \MM)$ and we have the inclusion $T(p; \MM) \subseteq \widehat{T}(p;\MM)$. Let us prove the reverse inclusion. In \cite[Theorem\,4.6]{Gallier} we find equivalent characterizations of a smooth manifold. In particular, we have that there is an open subset $U\subseteq \VV$ with $p\in U$, a finite dimensional real inner product space $\WW$, and a smooth submersion $f: U \to \WW$ (that is, $f$ is continuously differentiable in $U$ and the derivative $Df(a)$ is surjective for all $a\in U$) such that
    \begin{equation}\label{subm}
        \MM \cap U = \set{x \in U}{f(x) = 0}.
    \end{equation}
    Then, by applying \cite[Proposition\,5.38]{Lee} to the characterization \eqref{subm}, we have that the tangent space to $\MM$ at $p$ coincides with the kernel of $Df(p)$. That is,
    \begin{equation}\label{tangentdev}
        T(p; \MM) = \set{h \in \VV}{Df(p)h = 0}.
    \end{equation}
    Now, let $h\in \widehat{T}(p; \MM)$. If $h=0$ then $h\in T(p;\MM)$ because $T(p;\MM)$ is a linear subspace. Suppose that $h \neq 0$. Then, there exist sequences $\{x^k\}\subset\MM$, $\{t_k\}\subset\mathbb{R}_+$ such that $x^k\to p$, $t_k\downarrow 0$ and $h^k:=\frac{x^k-p}{t_k}\to h$. The definition of $Df(p)$ says that
    \begin{equation}\label{derivdef}
        \lim_{x\to p} \frac{\Vert f(x)-f(p)-Df(p)(x-p) \Vert}{\Vert x-p \Vert} = 0.
    \end{equation}
    Observe that $f(p)=0$ and $f(x^k)=0$ because $x^k = p + t_kh^k \in \MM \cap U$ for $k$ large enough. Then, by applying \eqref{derivdef} to the sequence $\{x^k\}$ we have
    \begin{equation*}
        0 = \lim_{k\to\infty} \frac{\Vert Df(p)h^k \Vert}{\Vert h^k \Vert} = \frac{\Vert Df(p)h \Vert}{\Vert h \Vert}.
    \end{equation*}
    Thus, $Df(p)h = 0$ and from \eqref{tangentdev}, we deduce $h \in T(p;\MM)$. We have proved $T(p;\MM) \supseteq \widehat{T}(p; \MM)$.
    
    $(b)$ $\widehat{N}(p; \MM)$ is the polar cone of $\widehat{T}(p;\MM)$ because of \eqref{tanchar}. Moreover, from the previous item, the polar cone of $\widehat{T}(p;\MM)$ is the orthogonal space of $T(p;\MM)$. Thus, the result follows.

    $(c)$ We know that $N(p;\MM)\supseteq\widehat{N}(p;\MM)$. Let us prove the reverse inclusion. Let $v\in N(p;\MM)$. Then, there exists sequences $\{x^k\}\subset \MM$ and $\{v^k\}\subset \VV$ such that $v^k\in \widehat{N}(x^k;\MM)$, $x^k\to p$ and $v^k\to v$. Because of item $(b)$ and from \eqref{tangentdev} we have that 
    \begin{equation}\label{comlink}
        v^k\in \big[ {\rm Ker}(Df(x^k)) \big]^\perp = {\rm Im}(Df(x^k)^\ast),
    \end{equation}
    for $k$ large enough and $f$ is given as in the proof item $(a)$. Since $Df(x^k)$ is surjective we have that $Df(x^k)^\ast$ is injective. Thus, 
    \[ \{Df(x^k)^\ast c_1, \ldots, Df(x^k)^\ast c_m\} \]
    is a basis for ${\rm Im}(Df(x^k)^\ast)$, where $\{c_1,\ldots,c_m\}$ is a fixed orthonormal basis of $\WW$. Then, from \eqref{comlink} we have that there exist sequences of scalars $\{\alpha^k_i\}$ for $i=1,\ldots,m$, such that
    \[ v^k = \alpha^k_1 Df(x^k)^\ast c_1 + \cdots + \alpha^k_m Df(x^k)^\ast c_m. \]
    From the continuity of $Df$, we have $Df(x^k)^\ast c_i \to Df(p)^\ast c_i$ for all $i=1,\ldots,m$. Furthermore, $\{Df(p)^\ast c_1,\ldots,Df(p)^\ast c_m\}$ is a basis for ${\rm Im}(Df(p)^\ast)$. This fact together with the convergence of $\{v^k\}$ ensure that  $\{\alpha^k_i\}$ is bounded. Thus, by passing to a subsequence, we may assume that  $\alpha_i^k\to\alpha_i$ for every $i=1,\ldots,m$, where $\alpha_1, \ldots, \alpha_m$ are appropriate scalars. Therefore,
    \[ v^k \to \alpha_1 Df(p)^\ast c_1 + \cdots + \alpha_m Df(p)^\ast c_m = v, \]
    where the last equality is due to the uniqueness of limits. This implies $v\in {\rm Im}(Df(p)^\ast)$. Then, $N(p;\MM)\subseteq {\rm Im}(Df(p)^\ast) = \widehat{N}(p;\MM)$.
\end{proof}

For $u \in \Rn$, $u^{\downarrow}$ is the vector obtained from $u$ by arranging its entries in non-increasing order. For $u,v\in\Rn$, one says that $u$ is majorized by $v$, in symbol $u\prec v$, if $u^{\downarrow}_1+\cdots+ u^{\downarrow}_k\leq v^{\downarrow}_1+\cdots+ v^{\downarrow}_k$ for every $k=1,\ldots,n-1$, and $u^{\downarrow}_1+\cdots+ u^{\downarrow}_n=v^{\downarrow}_1+\cdots+ v^{\downarrow}_n$ (cf. \cite[Definition\,1.A.1]{MOA}). A function $f : \Rn \to \mathbb{R}$ is said to be Schur-convex provided $u \prec v$ implies $f(u) \leq f(v)$ for $u, v \in \Rn$. Furthermore, $f$ is said to be strictly Schur-convex provided $u \prec v$ and $u^\downarrow\neq v^\downarrow$ imply $f(u) < f(v)$ for $u, v \in \Rn$. 

The following lemmas will be useful in the sequel.
\begin{lemma} \label{lem:strictSchur}
    Let $f : \Rn \to \mathbb{R}$ be a strictly convex symmetric function/norm. Then, $f$  is strictly Schur-convex.
\end{lemma}
\begin{proof}
The proof for the case of $f$ to be a strictly convex symmetric function is given in \cite[Theorem\,3.C.2.c]{MOA}. Suppose that $f$ is a strictly convex symmetric norm. Let $u,v\in\mathbb R^n$ be such that $u\prec v$ and $u^\downarrow\neq v^\downarrow$. Observe that $f$ is Schur-convex because it is a convex symmetric function (cf. \cite[Proposition\,3.C.2]{MOA}). Then, $f(u)\leq f(v)$. We need to prove that this inequality is strict. Reasoning by contradiction, suppose that $f(u)=f(v)$. Since $u\prec v$, we have that $u$ is in the convex hull of $\{Pv : P \in \Sigma_n\}$ with $\Sigma_n$ denoting the set of all permutation matrices of order $n$ (cf. \cite[Theorem\,2.B.6]{MOA}). From the symmetry of $f$ we have $f(Pv)=f(v)$ for all $P \in \Sigma_n$. Hence, $u$ and $Pv$, for $P \in \Sigma_n$, are in the sphere of radius $f(u)$, with respect to the norm $f$, and $u$ is a convex combination of $Pv$, for $P \in \Sigma_n$. Then, the strictly convex norm property of $f$ implies that $u = Pv$ for some $P \in \Sigma_n$. This contradicts the fact that $u^\downarrow\neq v^\downarrow$.
\end{proof}

\begin{lemma} \label{lem:majorization_sum}
    \normalfont{(\!\!\cite[Theorem\,4.5]{Gowda2019})} For any Euclidean Jordan algebra $\VV$ and for $x, y \in \VV$, the following majorization relation holds.
    \begin{equation} \label{eq:majorization_sum}
        \lambda(x + y) \prec \lambda(x) + \lambda(y).
    \end{equation}
\end{lemma}

In \cite{Baes}, Baes observed that for a spectral function $F : \VV \to \mathbb{R}$ with its corresponding symmetric function $f : \Rn \to \mathbb{R}$, the convexity of $f$ can be transfered to $F$. Later, in \cite{JG2017}, Jeong and Gowda demonstrated that convexity can indeed be transferred between $f$ and $F$ in both directions. The following propositions present analogous results for strictly convex function and strictly convex norm.

\begin{proposition} \label{prop:strictconvexity}
    Let $ F :\VV \to \mathbb{R}$ be a spectral function with the corresponding symmetric function $f:\Rn \to \mathbb{R}$ such that $F = f \circ \lambda$. In this case, $F$ is strictly convex if and only if so is $f$. 
\end{proposition}
\begin{proof}
    Suppose $f$ is strictly convex. Since $f$ is (strictly) convex, it is continuous and so is $F$. Thus, it is enough to show that the midpoint strict convexity holds for $F$. To see this, take $x, y \in \VV$ with $x \neq y$. Note that $f$ is Schur-convex as it is symmetric and (strictly) convex. Hence, we see from the Schur-convexity of $f$ applied to \eqref{eq:majorization_sum}, that
    \[ f \Big( \frac{\lambda(x+y)}{2} \Big) \leq f \Big( \frac{\lambda(x) + \lambda(y)}{2} \Big). \]
    Now, if $\lambda(x) \neq \lambda(y)$, the strict convexity of $f$ further implies
    \begin{align*}
        F \Big( \frac{x+y}{2} \Big) 
        = f \Big( \frac{\lambda(x+y)}{2} \Big) 
        & \leq f \Big( \frac{\lambda(x) + \lambda(y)}{2} \Big) \\
        & < \frac{f(\lambda(x)) + f(\lambda(y))}{2}
        = \frac{F(x) + F(y)}{2}. 
    \end{align*}
    On the other hand, suppose $\lambda(x) = \lambda(y)$. As $x \neq y$, $x$ and $y$ cannot strongly operator commute. Hence,  
    \begin{equation} \label{eq:assumption2}
        \lambda(x + y) \neq \lambda(x) + \lambda(y).
    \end{equation}
    Now, from \eqref{eq:majorization_sum} and \eqref{eq:assumption2} we deduce that $f(\lambda(x+y)) < f(\lambda(x) + \lambda(y))$ because $f$ is strictly Schur-convex as stated in Lemma\,\ref{lem:strictSchur}. Thus, in this case, we still have
    \begin{align*}
        F \Big( \frac{x+y}{2} \Big) 
        = f \Big( \frac{\lambda(x+y)}{2} \Big) 
        & < f \Big( \frac{\lambda(x) + \lambda(y)}{2} \Big) \\
        & = \frac{f(\lambda(x)) + f(\lambda(y))}{2}
        = \frac{F(x) + F(y)}{2}. 
    \end{align*}
    In any case, we have the desired strict inequality.

    For the converse, suppose that $F$ is strictly convex and take $u, v \in \Rn$ such that $u \neq v$ and a scalar $0<t<1$. Define $w := tu + (1-t)v$. For a fixed Jordan frame $\{e_1, \ldots, e_n\}$, define
    \[ x := \sum_{i=1}^{n} u_i e_i, \quad y := \sum_{i=1}^{n} v_i e_i, \quad \text{and} \quad z := \sum_{i=1}^{n} w_i e_i. \]
    It is easy to see that $x \neq y$ and $z = tx + (1-t)y$. Since $\lambda(z)$ is some permutation of $w$ and $f$ is symmetric, we have $F(z) = f(\lambda(z)) = f(w)$. Thus,
    \begin{align*}
        f \big( tu + (1-t)v \big) = f(w) = F(z) 
        &= F(tx + (1-t)y) \\
        & < tF(x) + (1-t)F(y) = tf(u) + (1-t)f(v).
    \end{align*}
    This completes the proof.
\end{proof}

\begin{proposition} \label{prop:norm}
    Let $ F :\VV\to\mathbb{R}$ be a spectral function with $f:\Rn\to\mathbb{R}$ as its associated symmetric function so that $F = f \circ \lambda$. In this setting, $F$ is a strictly convex norm on $\VV$ if and only if $f$ is a strictly convex norm on $\Rn$. 
\end{proposition}
\begin{proof}
    It is not difficult to show that $F$ is a norm on $\VV$ if and only if $f$ is a norm on $\Rn$. Thus, it is left to show that (norm) strict convexity is transferred between $F$ and $f$. 
    
    For the rest of the proof, assume first that $f$ is a strictly convex norm on $\Rn$. Let $x, y \in \VV$ be such that $x \neq y$ and $F(x) = F(y) = 1$. Consider first the case that $\lambda(x) \neq \lambda(y)$. Since $f(\lambda(x)) = f(\lambda(y)) = 1$ and $f$ is (norm) strictly convex, we see that
    \begin{equation} \label{eq:strictlyconvexnorm_1}
        f (\lambda(x) + \lambda(y)) < 2.
    \end{equation}
    Since $f$ is symmetric and convex as a function, it is Schur-convex. Thus, Schur-convexity of $f$ together with \eqref{eq:majorization_sum} implies
    \[ F(x + y) = f(\lambda(x + y)) \leq f(\lambda(x) + \lambda(y)) < 2. \]
    In the case that $\lambda(x) = \lambda(y)$, we note that $x$ and $y$ cannot strongly operator commute; hence \eqref{eq:assumption2} holds. From \eqref{eq:majorization_sum} and \eqref{eq:assumption2} we have $f(\lambda(x+y)) < f(\lambda(x) + \lambda(y))$ because $f$ is strictly Schur-convex as stated in Lemma\,\ref{lem:strictSchur}. Thus,
    \[ F(x + y) = f(\lambda(x+y)) < f(\lambda(x) + \lambda(y)) = f(2 \lambda(x)) = 2 f(\lambda(x)) = 2. \]
    Considering all the cases, we deduce that $F$ is a strictly convex norm on $\VV$.

    Conversely, assume that $F$ is a strictly convex norm on $\VV$. Take $u, v \in \Rn$ such that $u \neq v$ and $f(u) = f(v) = 1$. For a fixed Jordan frame $\{e_1, \ldots, e_n\}$, define
    $ x = \sum_{i=1}^{n} u^{\downarrow}_i e_i$ and $ y = \sum_{i=1}^{n} v^{\downarrow}_i e_i. $
    Since $\lambda(x)$ is some permutation of $u$ and $f$ is symmetric, we have $F(x) = f(\lambda(x)) = f(u) = 1$ and similarly, we have $F(y) = 1$.
    Also, note that $x$ and  $y$ strongly operator commute by our construction. Thus, we see that
    \[ \lambda(x+y) = \lambda(x) + \lambda(y) = u^{\downarrow} + v^{\downarrow}. \]
    Now, we have that $(u + v)^{\downarrow} \prec u^{\downarrow} + v^{\downarrow}$ (cf. \cite[Theorem 6.A.1.c]{MOA}). Then, the Schur-convexity of $f$ implies
    $ f((u+v)^{\downarrow}) \leq f(u^{\downarrow} + v^{\downarrow}). $
    It follows that
    \[ f(u + v) = f((u+v)^{\downarrow}) \leq f(u^{\downarrow} + v^{\downarrow}) = f(\lambda(x+y)) = F(x + y) < 2, \]
    where the last inequality follows from the norm strict convexity of $F$ together with the fact that $x \neq y$ and $F(x) = F(y) = 1$. This completes the proof.
\end{proof}

\begin{proof}[Proof of Lemma\,\ref{monot}.] Recall that $F=f\circ\lambda$.
     Let $x\in \VV$ and $v\in\partial F(x)$. From \eqref{subdifspec} and from the subdifferential definition \eqref{equal} we get
    \begin{equation}\label{subdiff}
        \ip{\lambda(v)}{u - \lambda(x)} \leq f(u) - f(\lambda(x)),
    \end{equation}
    for all $u\in\Rn$. Consider indices $i$ and $j$ with $\lambda_i(x) > \lambda_j(x)$. Let $P$ denote a permutation matrix that interchanges the $i$ and $j$ entries in $\lambda(x)$ and leaves all other entries intact. Put $\mu(x) := P \lambda(x)$ so that $\lambda(x)$ and $\mu(x)$ differ only in the $i$ and $j$ entries; hence neither of which is a positive multiple of the other. We now claim that
    \begin{equation}\label{strict ineq}
         \ip{\lambda(v)}{\mu(x) - \lambda(x)} < 0.
    \end{equation}
    Indeed, if $F$ is strictly convex, then the inequality \eqref{subdiff} holds strictly for $u \in \Rn$ with $u\neq \lambda(x)$ and, in particular, for $u = \mu(x)$. Then, by replacing this value in \eqref{subdiff} and because of $f(\mu(x))=f(\lambda(x))$, we obtain \eqref{strict ineq}. 
    
    Suppose now that $F$ is a strictly convex norm. Since $\mu(x)$ and $\lambda(x)$ are nonzero and neither of which is a positive multiple of the other, the unit vectors $\mu(x)/f(\lambda(x))$ and $\lambda(x)/f(\lambda(x))$ are distinct. Hence, since $f$ is a strictly convex norm (cf. Proposition\,\ref{prop:norm}), we deduce that
    \begin{equation}\label{nstrict}
        f(\mu(x) + \lambda(x))<2f(\lambda(x)).
    \end{equation}
    By substituting $u = \frac{1}{2}\mu(x) + \frac{1}{2}\lambda(x)$ in \eqref{subdiff} and by combining the ensuing inequality with \eqref{nstrict}, we obtain \eqref{strict ineq}. The claim follows. Now, the strict inequality \eqref{strict ineq} is equivalent to $(\lambda_i(x)-\lambda_j(x))(\lambda_i(v)-\lambda_j(v))>0$. Therefore, because of $\lambda_i(x)>\lambda_j(x)$, we conclude $\lambda_i(v)>\lambda_j(v)$.
\end{proof}

\begin{proof}[Proof of Proposition\,\ref{prop:transitive}.]
    Since $a$ and $b$ strongly operator commute, there exists a common Jordan frame $\{e_1, \ldots, e_n\}$ such that
    \begin{equation}\label{comuv}
        a = \sum_{i=1}^n \lambda_i(a) e_i \quad \text{and} \quad b = \sum_{j=1}^n \lambda_j(b) e_j.
    \end{equation}
    Let $\alpha_1 > \alpha_2> \cdots > \alpha_p$ be the distinct values of $\lambda_1(a), \ldots, \lambda_n(a)$, and let $\beta_1 > \beta_2 > \cdots > \beta_q$ be the distinct values of $\lambda_1(b), \ldots, \lambda_n(b)$. From these distinct eigenvalues of $u$ and $v$, we can construct two partitions of $\{1, \ldots, n\}$ as follows: 
    \[ \mathcal{I} := \{I_1, \ldots, I_p \} \quad \text{and} \quad \mathcal{J} := \{ J_1, \ldots, J_q \}, \]
    where $I_r = \set{1 \leq i \leq n}{\lambda_i(a) = \alpha_r}$ for all $r = 1, \ldots, p$, and $J_s = \set{1 \leq j \leq n}{\lambda_j(b) = \beta_s}$ for all $s = 1, \ldots, q$. Write $E_r = \sum_{i \in I_r} e_i$ for all $r = 1, \ldots, p$, and $F_s = \sum_{j \in J_s} e_j$ for all $s = 1, \ldots, q$. Then, \eqref{comuv} can be rewritten as
    \begin{equation} \label{spectraluniq}
        a = \sum_{r=1}^{p} \alpha_r E_r \quad \text{and} \quad 
        b = \sum_{s=1}^{q} \beta_s F_s.
    \end{equation}
    From the assumption that $\lambda_i(a) > \lambda_j(a)$ implies $\lambda_i(b) > \lambda_j(b)$ for all $i,j$, we see that $p \leq q$ and the partition $\mathcal{J}$ is finer than the partition $\mathcal{I}$. Furthermore, each set of the partition $\mathcal J$ is contained in some set of the partition $\mathcal I$. Thus, there exists a partition $\{T_1, \ldots, T_p\}$ of $\{1, \ldots, q\}$ such that for every $r = 1, \ldots, p$,
    $ \sum_{s \in T_r} F_s = E_r. $
    Now, since $b$ and $c$ operator commute, there exists a Jordan frame $\{e'_1, \ldots, e'_n\}$ and scalars $\gamma_1,\ldots,\gamma_n$ such that 
    \begin{equation} \label{commxv}
        b = \sum_{i=1}^n \lambda_i(b) e'_i \quad \text{and} \quad c = \sum_{i=1}^n \gamma_j e'_j.
    \end{equation}
    Observe that we have two different-looking ways of writing the spectral decomposition of $b$ as in \eqref{comuv} and \eqref{commxv}. However, since the representation of $b$ in \eqref{spectraluniq} is unique (cf. \cite[Theorem\,III.1.1]{FK}), we must have
    \[ F_s = \sum_{j \in J_s} e_j = \sum_{j \in J_s} e'_j. \]
    Hence, $a$ can be written as
    \begin{equation} \label{domposea}
        a = \sum_{r=1}^{p} \alpha_r E_r = \sum_{r=1}^{p} \alpha_r \sum_{s \in T_r} F_s = \sum_{r=1}^{p} \alpha_r \sum_{s \in T_r} \sum_{j \in J_s} e'_j = \sum_{r=1}^{p} \sum_{s \in T_r} \sum_{j \in J_s} \alpha_r e'_j.
    \end{equation}
    Since the right-hand side of \eqref{domposea} uses all the $e'_j$s exactly once, it is a spectral decomposition of $a$. Thus, $a$ and $c$ can be decomposed simultaneously with a common Jordan frame.  Hence, $a$ and $c$ operator commute.
\end{proof}

\end{document}